\documentclass[accepted]{uai2026} % for initial submission
\usepackage[british]{babel}
\usepackage{natbib} % has a nice set of citation styles and commands
\usepackage{mathtools} % amsmath with fixes and additions
\usepackage{booktabs} % commands to create good-looking tables
\usepackage{hyperref}       % hyperlinks
\usepackage{url}            % simple URL typesetting
\usepackage{amsfonts}       % blackboard math symbols
\usepackage{nicefrac}       % compact symbols for 1/2, etc.
\usepackage{microtype}      % microtypography
\usepackage[dvipsnames]{xcolor}         % colors
\usepackage{comment}

\usepackage{subcaption}
\usepackage{amsmath}
\usepackage{amsthm}
\usepackage{amssymb}
\usepackage{bm}
\usepackage[inline]{enumitem}
\mathtoolsset{showonlyrefs}

\def\given{\,|\,}

\newtheorem{theorem}{Theorem}
\newtheorem{proposition}{Proposition}

\newtheorem{lemma}{Lemma}

\theoremstyle{remark}
\newtheorem{remark}{Remark}

\newtheorem*{assumptions*}{Assumptions}

\title{From Bayes' Rule to Bayes Rules: \\
Optimal Information Processing and Axiomatic Foundations Beyond Probability}

\author[1]{Jeremie Houssineau}
\author[2]{Badr-Eddine Ch\'erief-Abdellatif \thanks{Corresponding author:  \href{mailto:badr-eddine.cherief-abdellatif@cnrs.fr}{badr-eddine.cherief-abdellatif@cnrs.fr}}}
\affil[1]{%
    Nanyang Technological University, Singapore
}
\affil[2]{%
    CNRS, LPSM, Sorbonne Universit\'e, France
}

\begin{document}

\maketitle

\begin{abstract}
This paper develops principled updating rules for possibilistic inference, where uncertainty about a fixed parameter is represented by a possibility function, the maxitive analogue of a probability distribution, and comparisons are made pointwise via a partial order. From two complementary foundations, an information-conservation viewpoint and an axiomatic viewpoint, we derive the same canonical update: the posterior is the prior-likelihood product followed by supremum normalisation. The two derivations agree for an arbitrary loss, differing only in where the learning-rate parameter enters. This parameter controls epistemic strength and is not identifiable from the normalising evidence alone, clarifying the role of analogous learning-rate parameters in generalised Bayesian updating.
\end{abstract}

%%%%
\section{Introduction}

Bayes' rule occupies a unique position in statistical inference: it is at once an operational framework for updating beliefs and a benchmark against which alternative update rules are judged. Yet both classical statistics and modern machine learning routinely depart from standard Bayesian updating. Classical examples include inverse-probability methods and alternative postdata constructions, while modern examples include loss-based ``generalised Bayes'' updates, tempered posteriors, and evidence-fusion schemes in which uncertainty is not naturally additive. These departures raise a basic question that is easy to state but conceptually difficult: \emph{what exactly makes Bayes' rule the ``right'' way to process information, and how does that answer depend on the mathematical representation of information?}

One way to revisit this question is to treat inference as \emph{information processing}. Starting from two inputs,
\begin{enumerate*}[label=\roman*)]
\item prior information about an unknown but fixed quantity of interest and
\item information supplied by data through a likelihood or loss,
\end{enumerate*}
one asks for a principled information processing rule (IPR) that outputs a postdata representation of the state of knowledge. In the probabilistic setting, such a perspective is often associated with the work of \citet{zellner1988optimal} on the information conservation principle (ICP), where Bayes' rule arises as the unique IPR that conserves a chosen scalar measure of information. In parallel, the framework of \citet{bissiri2016general} shows that, when data enter through losses rather than a data-generating mechanism (DGM), a Gibbs (exponentiated-loss) posterior is forced by coherence requirements on how losses should accumulate and how updating should behave under restriction.

In this work, we argue that these ``foundational'' stories are not exclusive to probability, but translate naturally when uncertainty about the parameter is represented by \emph{possibility functions}. Possibility theory \citep{zadeh1978fuzzy} is designed to model epistemic uncertainty about a fixed state of nature. A prior state of knowledge about $\theta\in\Theta$ is represented by a possibility function $\pi:\Theta\to[0,1]$ with $\sup_{\theta\in\Theta}\pi(\theta)=1$, where larger values indicate greater compatibility with the available information. Crucially, the set of possibility functions $\mathcal{F}(\Theta)$ on a set $\Theta$ is naturally ordered pointwise: one possibility function can directly be said to be (weakly) more informative than another without first compressing them into a single number such as the (cross-)entropy. This order structure has two immediate consequences for information processing:
\begin{enumerate*}[label=\roman*)]
\item it becomes meaningful to distinguish \emph{information loss} from \emph{information creation} at the level of information objects themselves, and
\item combining independent informational inputs naturally corresponds to multiplication followed by supremum normalisation, yielding a maxitive analogue of the familiar prior-likelihood product.
\end{enumerate*}

Our central message is that, under two conceptually distinct viewpoints, the same possibilistic update rule emerges as inevitable:
\begin{align}
\label{eq:maxBayesRule}
\pi(\theta\given y) = \frac{\pi(\theta) L(\theta\given y)}{\sup_{\vartheta\in\Theta}\pi(\vartheta) L(\vartheta\given y)},
\end{align}
with $L(\theta\given y)$ a likelihood and $\pi(\theta)$ a prior possibility function. When the likelihood $L(\theta\given y)$ is a genuine sampling distribution $p_{\theta}(y)$, this update corresponds to a sup-normalised prior-likelihood product. When the input is a generic loss $\ell(\theta,y)$, the same structure appears with $L(\theta\given y)\propto \exp(-w\ell(\theta,y))$ for a learning-rate parameter $w \geq 0$. In this sense, Bayes' rule is not a single formula but a member of a family of \emph{Bayes' rules} indexed by the underlying calculus of information (additive vs.\ maxitive) and by the nature of the data-to-parameter link (probabilistic likelihood vs.\ loss-derived evidence).

\paragraph{Contributions.}
The paper is organised around two derivations of the possibilistic Bayes' rule, each mirroring a classical probabilistic counterpart but revealing structural differences that are specific to ordered, sup-normalised information objects.
\begin{itemize}
\item \textbf{View 1 (optimal information processing).} We reformulate information conservation without appealing to scalar cross-entropies. By introducing a pointwise information score and enforcing additivity of information under independent conjunction, we obtain a canonical ``$-\log$'' information map on possibility values. This yields an information imbalance \emph{function} (rather than a scalar), making information loss/creation genuine phenomena. Imposing conservation in the induced partial order identifies the sup-normalised prior-likelihood product as the unique lossless and non-creative IPR.
\item \textbf{View 2 (axiomatic approach).} We propose a possibilistic analogue of the coherence axioms of \citet{bissiri2016general}, phrased for updates that map a prior possibility function and a loss into a posterior possibility function. These axioms force a sup-normalised Gibbs form $\pi(\theta\given y) \propto \pi(\theta)\exp\big(-w\ell(\theta, y)\big)$, hence deriving the same structural rule without presupposing it.

\item \textbf{Unification and interpretation.} We show that the two views yield the \emph{same} one-parameter family of updates for an arbitrary loss, with the learning rate $w$ exogenous to both derivations (an encoding choice in View~1, an
axiom-underdetermined constant in View~2). This clarifies the role of $w$ as an epistemic-strength calibration parameter rather than a data-identifiable quantity, and explains why evidence-based learning of $w$ is ill-posed in the possibilistic setting. The log-loss, $\ell(\theta,y)=-\log p_\theta(y)$, recovers the more familiar (tempered) Bayesian update as a special case.
\end{itemize}

%%%%
\section{Related work}

Beyond \citet{bissiri2016general}, foundational perspectives on generalised Bayes include divergence-/scoring-rule motivations \citep{jewson2018principles, knoblauch2022optimization}, quasi-posteriors based on generic criterion functions \citep{chernozhukov2003mcmc}, and work on calibrating the loss scale/learning rate for misspecification-robustness \citep{grunwald2017inconsistency,syring2019calibrating,lyddon2019general}.
In the information-processing tradition related to \citet{zellner1988optimal}, Bayes updating is also closely related to information-theoretic projection principles (maximum entropy / minimum cross-entropy and minimum discrimination information) and their axiomatic characterisations \citep{jaynes1957information, shore1980axiomatic, csiszar1975divergence}.

Bayes-like updates in possibility theory have long been studied through \emph{conditional possibility} and notions of independence/noninteraction, leading to several competing conditioning rules and systematic treatments \citep{hisdal1978conditional, decooman1997conditional}, with \citet{walley1999coherence} providing a behavioural/coherence perspective on such conditioning rules. A likelihood-based approach for possibility measures is developed by \citet{dubois1997semantics}. Our contribution differs in that we derive a \emph{sup-normalised prior-likelihood product} as the unique information-processing rule compatible with (loss-based) updating and information conservation, rather than positing a conditioning rule \emph{a priori}. Our approach aligns with recent imprecise-probabilistic perspectives in which sup-normalisation emerges as the appropriate calibration step when combining a likelihood with partial prior information \citep{martin2026possibilistic}. More broadly, this perspective aligns with the emerging Imprecise Probabilistic Machine Learning (IPML) literature on learning and decision-making under imprecise information such as \citet{caprio2024credal,singh2024domain} or \citet{chau2025integral}. A survey of possibility theory and its relation to probability is given by \citet[Ch.~5]{dubois2015possibility}, on which we draw to justify combining the two: our practical focus is the mixed setting of a probabilistic likelihood $p_\theta(y)$ with a possibilistic prior, i.e.\ the data treated as random while beliefs about $\theta$ remain possibilistic.

Among imprecise-probability models, we work with possibility measures because their \emph{pointwise} representation by a possibility function is what enables the partial order on $\mathcal{F}(\Theta)$ underlying View~1, where loss and creation become pointwise phenomena rather than a scalar summary. Set-function models, e.g., coherent lower previsions \citep{walley1991statistical}, the imprecise Dirichlet model \citep{bernard2005introduction}, p-boxes, probability intervals, or distortion models, also support generalised Bayes updating and can be ordered by credal-set inclusion, but they admit no comparably direct pointwise order, so an information-conservation analysis there would require a different notion of imbalance.

Most closely related is the companion paper of \citet{singh2025maxitive} which develops a maxitive analogue of the Donsker--Varadhan variational formula and applies it to \emph{possibilistic variational inference}. The two are complementary: \citet{singh2025maxitive} is computational and objective-driven (variational characterisations and tractable candidate families), whereas the present paper is foundational, axiomatising the \emph{update rule itself} through the information-processing and axiomatic viewpoints and explaining why the sup-normalised prior-likelihood product is canonical. The viewpoints connect directly: our no-loss/no-creation characterisation of the posterior (Section~\ref{sec:dual-guarantees}) coincides with their dual consistency bounds, and we recover their variational uniqueness result as an information-conservation statement. This programme has already reached deep learning: \citet{ni2026possibilistic} instantiate the possibilistic posterior on the probability simplex, learning a Dirichlet possibility function that approximates a supremum-projected posterior by minimising the maxitive divergence of \citet{singh2025maxitive}, and obtain second-order uncertainty predictors competitive with the state of the art at the cost of a single forward pass, a downstream application of the very update rule characterised here.

%%%%
\section{Possibility Theory}
\label{sec:possibility_theory}

Let $\theta^*$ be an unknown parameter in a set $\Theta$, which we model by an \emph{uncertain variable} $\bm{\theta}$, the analogue of a random variable in possibility theory. The prior uncertainty about $\bm{\theta}$ (equiv.\ $\theta^*$) is represented by a possibility function $\pi(\theta)$ in $\mathcal{F}(\Theta) = \{f:\Theta \to [0,1] \, : \, \sup_{\theta\in \Theta} f(\theta)=1 \}$, which directly models the available information about $\bm{\theta}$. The possibility function $\pi(\theta)$ only \emph{describes} the information about $\bm{\theta}$ as there is no canonical epistemic uncertainty (as opposed to the probabilistic case where an underlying probability measure, often denoted $\mathbb{P}$, defines the \emph{true} probability of any suitable event). We consider that information about $\bm{\theta}$ is available through an experiment which yields an observation $y$, with corresponding likelihood $L(\theta \given y)$. We will consider multiple cases for the likelihood, which can be defined as
\begin{enumerate*}[label=\roman*)]
    \item a probability distribution $p_{\theta}(y)$ in $y$,
    \item a possibility function $f_{\theta}(y)$ in $y$, and
    \item a loss $\ell(\theta, y)$, via $\exp(- \ell(\theta, y))$.
\end{enumerate*}
The last case overlaps with the first two cases by setting $\ell(\theta,y)=-\log p_{\theta}(y)$ or $\ell(\theta,y)=-\log f_{\theta}(y)$. Conversely, if a loss $\ell(\theta, y)$ is specified first, it can be transformed into a possibility function whenever it is bounded below. Writing $m(y)\doteq\inf_{\theta\in\Theta}\ell(\theta,y)$, the relative evidence
\begin{align}
    \label{eq:relative-evidence}
    h_y(\theta) \doteq \exp\big(-(\ell(\theta,y)-m(y))\big) \in (0,1]
\end{align}
is a possibility function. This transformation is meaningful since both losses and possibility functions capture information, whereas transforming a loss into a probability distribution (when possible) does not guarantee that the obtained probability will have any meaning as a DGM. Unlike a genuine sampling density, however, a loss carries no intrinsic scale: for any strength $w\ge 0$ the tempered evidence $h_y^{w}\in\mathcal F(\Theta)$ is an equally valid encoding of the same information, so $w$ is a modelling input rather than a property of the loss. We make this strength explicit from the outset and return to its role below.

\paragraph{Priors, regularisers, and provenance.} The same scaling question applies to the prior when prior information comes from a regulariser. When the prior is specified directly, or obtained as a relative likelihood
$\pi(\theta)=p_0(\theta)/\sup_{\vartheta}p_0(\vartheta)$ from a genuine model
\citep{walley1999upper,denoeux2014likelihood}, its scale is fixed by its provenance, and tempering it to $\pi^\beta$ with $\beta\neq1$ departs from that model. When the prior instead encodes a penalty or regulariser $R:\Theta\to[0,\infty)$ through $\pi(\theta) = \exp(-\beta(R(\theta)-\inf_{\vartheta}R(\vartheta)))$, the specificity of $\beta = 1$ vanishes and $\beta\ge0$ is a free strength parameter, exactly as $w$ is for the loss. Nothing in the formalism singles out the prior as the scale-fixed factor: prior and evidence enter \eqref{eq:maxBayesRule} symmetrically as multiplicative possibilistic terms, so freezing the prior's scale is a modelling stance. Unless otherwise stated, we assume that $\pi$ is specified directly and defer further discussion on this to Appendix~\ref{app:losses-regularisers}.

\paragraph{Independence.} For two uncertain variables $\bm{\theta}$ and $\bm{\theta}'$ represented jointly by a possibility function $f_{\bm{\theta},\bm{\theta}'}(\theta,\theta')$ on $\Theta \times \Theta'$, independence is characterised by the existence of two possibility functions $f_{\bm{\theta}}(\theta)$ and $f_{\bm{\theta}'}(\theta')$ such that $f_{\bm{\theta},\bm{\theta}'}(\theta,\theta') = f_{\bm{\theta}}(\theta)f_{\bm{\theta}'}(\theta')$ for all $(\theta,\theta') \in \Theta \times \Theta'$. Although this characterisation of independence resembles the probabilistic one, the fact that possibility functions are not induced by the uncertain variable they describe means that independence is a property of possibility functions rather than a property of uncertain variables. In a Bayesian context, the prior and likelihood can now be said to be independent when they are both possibility functions. This can improve conceptual transparency in epistemic-uncertainty and evidence-fusion settings, where ``independence'' is fundamentally about sources rather than randomness. This advantage comes with the responsibility to assess overlap between sources, otherwise beliefs could be unintentionally over-concentrated.

\paragraph{Non-identifiability of strength.}
Whenever an information source is a possibility function rather than a genuine sampling density (e.g., a possibilistic likelihood, the loss-based evidence $h_y$, or a regulariser prior), its strength is a free parameter, since possibility functions are closed under power. Indeed, $f^{w}\in\mathcal F(\Theta)$ for every $w\ge0$ and every $f\in\mathcal F(\Theta)$. It is tempting to place a prior on the strengths and learn them from data. However, the consistency
\begin{align*}
    c(y) & \doteq \sup_{\theta\in\Theta}\pi(\theta)^{\beta} h_y(\theta)^{w} \\
    & = \exp\Big(-\inf_{\theta\in\Theta}
      \big[\beta R(\theta)+w(\ell(\theta,y)-m(y))\big]\Big),
\end{align*}
with $R\doteq-\log\pi\ge0$, is non-increasing in both $\beta$ and $w$, so maximising the possibilistic evidence merely drives each admissible strength to its smallest admissible value. The strengths are therefore not identifiable from the evidence alone: scaling reflects epistemic strength, not data-generating variability, and, in short, one cannot learn how well one knows something. This is illustrated in Appendix~\ref{sec:sq-loss-misspec}. The phenomenon is familiar in generalised Bayesian inference \citep{bissiri2016general} for a loss-based likelihood; in the possibilistic setting it arises intrinsically and symmetrically, for every factor whose provenance does not fix its scale. Another example of this indeterminacy arises in decentralised possibilistic fusion \citep{houssineau2026decentralised}, where a possibility function $f$ may be split into independent shares $f^{w}$ and $f^{1-w}$ and fused back to $f$ exactly for any $w\in[0,1]$. In this setting, one need not \emph{choose} the discounting weight, just as here one cannot \emph{learn} it.

Thus the rest of the paper treats the prior $\pi$, the evidence possibility $h_y^w$, and the posterior $\pi(\cdot\given y)$ as elements of $\mathcal F(\Theta)$, all compared through the same pointwise order.

\paragraph{Existing results.} The update rule \eqref{eq:maxBayesRule} already shares a number of properties with its probabilistic counterpart such as a Bernstein--von Mises theorem \citep{hieu2025decoupling}. We do not focus on these properties in this work as they are mainly \emph{consequential} rather than \emph{characterising}, i.e., they start from \eqref{eq:maxBayesRule} instead of leading to it.

\section{View 1: Optimal Information Processing}
\label{sec:4}

\citet{zellner1988optimal} shows that the standard Bayesian posterior can be recovered from an IPR which transforms input information into output information. The argument relies on the introduction of a generic \emph{postdata} distribution as one of the components in the output information (along with the evidence), which is then compared to the input information, based on the prior (or \emph{antedata}) distribution and the likelihood. A detailed review of the argument by \citet{zellner1988optimal} is provided in Appendix~\ref{app:optimal_information_processing} for completeness.

\subsection{Our view on Zellner's analysis}
\label{sec:probabilisticIPR}

We begin by emphasizing a point of fundamental importance: in Zellner's ICP, what is being conserved is not information itself, but rather a \emph{numerical measure} of information, that is, an \emph{amount of information}. \citet{zellner1988optimal} himself frames the efficiency of an IPR in terms of output information measured in a suitably chosen metric, so that comparing input and output is metric-dependent by construction. Statements about information being ``lost'', ``created'', or ``conserved'' only become meaningful once this information is mapped to a scalar quantity through the choice of a metric. From this perspective, Zellner's ICP can be interpreted as a principle concerning the conservation of a specific numerical representation of information, rather than information itself as an abstract object.

A natural candidate for measuring information is Shannon entropy. However, Shannon entropy is an absolute quantity: it is not defined relationally between two informational objects. As such, it does not provide a direct mechanism for comparing input and output information arising from different sources. In particular, entropy alone does not induce a meaningful conservation principle linking these distinct objects. To overcome this difficulty, Zellner adopts the Shannon cross-entropy, which is a relative quantity defined with respect to a reference distribution. This choice renders the ICP mathematically well-posed and leads to a unique postdata distribution satisfying the conservation condition, namely the Bayesian posterior. Equivalently, Bayes' rule minimises the scalar functional $\Delta(\text{IPR})$, though this variational form is a reformulation of the conservation property rather than a primitive principle. Importantly, however, this uniqueness is a consequence of the chosen information measure rather than an intrinsic property of information processing itself.

Within Zellner's framework, the information loss $\Delta(\text{IPR})$ induced by an IPR can be rewritten as:
\begin{align}
\Delta(\text{IPR}) & \doteq \text{Input info}(\text{IPR}) - \text{Output info}(\text{IPR}) \\
& = \mathrm{KL}\left(q(\cdot \given y) \| \pi_{\textrm{add}}(\cdot \given y) \right) \geq 0 \, ,
\end{align}
where $q(\cdot \given y)$ is the postdata distribution produced by the IPR and $\pi_{\mathrm{add}}$ is the Bayesian posterior, the subscript ``add'' referring to the additive nature of the posterior probability. This result entails two immediate and somewhat paradoxical consequences, corresponding to the sign of $\Delta(\text{IPR})$ and to its case of equality. First, since $\Delta(\text{IPR}) \ge 0$, \textbf{no IPR can create extraneous information}: the output can never exceed the input, so even a rule that adds structure to the belief gains nothing in these terms. This already rules out Zellner's own suggestion that some IPRs might raise output above input. Second, since $\Delta(\text{IPR}) = 0$ only when $q(\cdot \given y) = \pi_{\textrm{add}}(\cdot \given y)$, \textbf{every IPR other than Bayes' rule strictly loses information}. While mathematically consistent within the adopted framework, this conflicts with common intuitions about informativeness: an IPR whose postdata belief is more concentrated around the true parameter value than the Bayesian posterior would usually be judged more informative, not less. The notion of information loss arising here is therefore tightly linked to coherence with the prior and likelihood, rather than to alternative notions of informativeness such as concentration, precision, or proximity to the truth.

We argue that these paradoxes are epistemological in nature. Comparing input and output information implicitly assumes that informational objects can be meaningfully ordered or compared. When information is represented by probability distributions, this assumption is problematic: the set $\mathcal{P}(\Theta)$ of probability distributions over $\Theta$ has no canonical pointwise informativeness order analogous to the possibility-function order. Any comparison between distributions must therefore be achieved through an externally chosen functional, such as entropy or cross-entropy. Consequently, the ICP does not compare information objects directly, but only compares scalar summaries derived from them. Therefore, the conclusions drawn from such comparisons are tied to the choice of information measure and cannot be regarded as intrinsic properties of information processing.

Our claim is that possibility theory provides a more appropriate framework for comparing informational objects directly. The set $\mathcal{F}(\Theta)$ of possibility functions over $\Theta$ is naturally partially ordered by
\[
f \preceq g \quad \Longleftrightarrow \quad \forall \theta \in \Theta,\; f(\theta) \leq g(\theta).
\]
This structure admits a greatest element (the constant function equal to one) which represents total ignorance, a notion that has no direct analogue in probability theory. Within this ordered setting, information can be compared without recourse to a scalar summary: one informational object can be directly said to contain more or less information than another.

In the following subsection, we show that when information is represented by possibility functions and the ICP is formulated in terms of the order $\preceq$, the possibilistic Bayes' rule emerges as the unique postdata belief that conserves information. Moreover, this framework allows for a clear identification of IPRs that genuinely induce information loss or information gain, without relying on arbitrary numerical measures.

\subsection{Possibilistic Bayes' Rule as Optimal Information Processing}

We now reformulate statistical inference as an \emph{information processing} problem within a possibilistic framework. As in the previous section, an \emph{information processing rule} (IPR) transforms given \emph{input information} into \emph{output information}. The crucial difference lies in the representation of uncertainty about the unknown parameter: uncertainty about the parameter $\theta$ is represented by \emph{possibility functions}, while the likelihood is loss-based, with a
strength parameter $w \geq 0$, as identified in Section~\ref{sec:possibility_theory} as an exogenous input rather than a feature of the loss.

The goal here is threefold. First, we show that within this framework, both information loss and information creation can occur. Second, we show that Bayes' rule emerges as the \emph{unique} IPR that neither loses nor creates information. Finally, we contrast this result with the probabilistic setting, where such distinctions were impossible by construction.

We consider a postdata possibility function $g(\theta \given y)$ with no assumption that it must follow Bayes' rule. Instead, it is entirely determined by the chosen IPR, whose optimality is to be characterised. In contrast with the probabilistic case, information is no longer summarised by scalar expectations, i.e.\ the cross-entropy. Instead, information is represented directly as a \emph{function over the parameter space}. 

\begin{assumptions*}[Pointwise information]
Let $S:(0,1]\to[0,\infty)$ be a scalar \emph{information score} on possibility values and define, for any positive function $f:\Theta\to(0,1]$, the pointwise information score
$\mathcal{I}[f]:\Theta\to[0,\infty)$ by
\[
\mathcal{I}[f](\theta) = S\bigl(f(\theta)\bigr).
\]
Assume:
\begin{enumerate}[label=A1.\arabic*,leftmargin=25pt]
\item \label{A1.1} (\emph{Order-consistency}) For any $0 < a < b \leq 1$, it holds that $S(a) > S(b)$.
\item \label{A1.2} (\emph{Normalisation}) $S(1)=0$.
\item \label{A1.3} (\emph{Additivity under independent conjunction}) For any $a,b\in(0,1]$, $S(ab)=S(a)+S(b)$.
\item \label{A1.4} (\emph{Regularity}) $S$ is continuous on $(0,1]$.
\end{enumerate}
\end{assumptions*}

Assumptions~\ref{A1.1}--\ref{A1.4} are the pointwise counterparts of the standard axioms characterising Shannon self-information, and lead to the same logarithmic map.

\begin{proposition}[Canonical pointwise information score]
\label{prop:canonical_pointwise_information}
Under Assumptions~\ref{A1.1}--\ref{A1.4}, there exists a constant $k>0$ such that for all $a\in(0,1]$,
\[
S(a) = - k \log a.
\]
Consequently, for any positive $f:\Theta\to(0,1]$,
\[
\mathcal{I}[f](\theta) = - k \log f(\theta), \qquad \theta\in\Theta.
\]
\end{proposition}

\begin{proof}
Assumption \ref{A1.3} implies that $S$ satisfies Cauchy's functional equation on $(0,1]$ under multiplication. By Assumption~\ref{A1.4}, the only solutions are of the form $S(a)=c\log a$ for some $c\in\mathbb{R}$. Assumptions~\ref{A1.1} and \ref{A1.2} yield $c<0$ and $S(1)=0$, hence $S(a)=k(-\log a)$ with $k=-c>0$.
\end{proof}

Accordingly, we define $\mathcal{I}[f] : \Theta \to \mathbb{R}$ for any positive function $f:\Theta\to(0,1]$ as the function $\mathcal{I}[f](\theta) = -\log f(\theta)$ for all $\theta \in \Theta$ (so lower possibility corresponds to higher information content), by considering the convention $k=1$ in Proposition~\ref{prop:canonical_pointwise_information}. This convention fixes the \emph{unit} in which information is measured and
applies identically to every possibility function; it does not fix the
\emph{scale} of any particular input. A possibility function built from a loss
carries its own strength inside the object being measured, so its information
content is a multiple $w$ of the (shifted) loss even though $\mathcal{I}$ itself uses $k=1$ (see below, and Appendix~\ref{app:losses-regularisers}). Information is thus no longer compressed into a single numerical quantity. Instead, it is preserved as a structured object that can be compared pointwise over $\Theta$.
When possibility functions take the value $0$, we interpret $-\log 0=+\infty$
and work with extended-real information scores, or equivalently restrict the
pointwise arguments to the positive support and recover zero posterior values by continuity.

\begin{remark}[On factorisation and conjunction]
The factorisation $f_{\bm{\theta},\bm{\theta}'}(\theta,\theta') = f_{\bm{\theta}}(\theta) f_{\bm{\theta}'}(\theta')$ and the additivity assumption (\ref{A1.3}) should not be read as claiming that the product is the only possible conjunction in possibility theory. Rather, they fix the \emph{information calculus} under which we operate: we work with the product t-norm as the conjunctive rule for combining independent information sources, which is precisely the regime where a log-additive information map is natural. Under an alternative conjunction $\otimes$, the corresponding additivity requirement would be formulated in terms of $\otimes$, and the induced information map and update rule would generally differ, yielding a different member of the family of ``Bayes' rules''. The present derivation is therefore a characterisation \emph{conditional} on adopting product conjunction for independent information. The induced information map and update rule change accordingly (the minimum, for instance, yields minitive conditioning $\pi(\theta\given y)\propto\min(\pi(\theta),L(\theta\given y))$) and the product is the case where the two views collapse into one rule.
We leave the full characterisation of such an alternative information calculus to future work.
\end{remark}

\paragraph{Phrasing Bayes' rule as combining information.} The operator
$\mathcal{I}[\cdot]$ takes possibility functions as arguments, but neither
$L(\theta \given y)$ nor $c(y)$ in \eqref{eq:maxBayesRule} is necessarily
a possibility function. We therefore rephrase \eqref{eq:maxBayesRule} so that all
components are possibilistic. Following Section~\ref{sec:possibility_theory}, let
$\ell(\theta, y)$ be any loss with $m(y) \doteq \inf_{\vartheta\in\Theta}
\ell(\vartheta,y)$ finite, and let the evidence possibility at strength $w\ge0$ be
\[
h_y^{w}(\theta) = \exp\big(-w(\ell(\theta, y)-m(y))\big)\in(0,1].
\]
Measured with $\mathcal I$, its information content is $\mathcal{I}[h_y^{w}](\theta) = w\,(\ell(\theta,y)-m(y))$: the loss \emph{is} the information carried by the evidence, up to the scale $w$. Unit strength $w=1$ is the natural default, being the choice under which the shifted loss coincides exactly with the information content in nats; any other $w$ rescales that information without changing its nature, which is precisely the non-identifiable strength of Section~\ref{sec:possibility_theory}. When $\ell(\theta, y)=-\log p_\theta(y)$, unit strength recovers the relative likelihood $h_y^{1}(\theta) = p_{\theta}(y)/\sup_{\vartheta} p_{\vartheta}(y)$ whenever the supremum is finite; when $\sup_{\vartheta} p_{\vartheta}(y)=\infty$, the loss-based definition remains meaningful as soon as $m(y)$ is finite (e.g., after restricting $\Theta$ or regularising $\ell$).

Introducing an uncertain variable $\bm{\theta}'$ described by $h_y^{w}$ alongside $\bm{\theta}$ described by $\pi$, and assuming $h_y^{w}$ and $\pi$ independent,
\begin{align}
\label{eq:BayesByCombination}
f(\theta \given \bm{\theta}' = \bm{\theta})
= \frac{h_y^{w}(\theta)\, \pi(\theta)}
       {\sup_{\vartheta \in \Theta} h_y^{w}(\vartheta)\, \pi(\vartheta)}
= \pi(\theta \given y).
\end{align}
The normalising constant $c(y) \doteq \sup_{\vartheta \in \Theta}
h_y^{w}(\vartheta) \pi(\vartheta)$ lies in $(0,1]$ and is itself a possibility, so $\mathcal{I}[\cdot]$ is now defined for every element of \eqref{eq:BayesByCombination}. It also makes clear that $c(y)$, the degree of consistency, is an output information absent from the postdata possibility function, as it is what is conditioned on. Henceforth we take $L(\theta \given y) \doteq h_y^{w}(\theta)$, so that \eqref{eq:maxBayesRule} coincides with \eqref{eq:BayesByCombination} and carries the strength $w$. Although \eqref{eq:BayesByCombination} has the same multiplicative shape as the usual Bayesian product identity, its semantics are different: $c(y)$ is not a marginal data-generating probability, but an information term associated with the compatibility of prior and evidence.

\paragraph{Information Conservation Principle (ICP).} As before, we postulate that a good IPR should conserve information, i.e.\ ``Input information = Output information''. An IPR satisfying this condition is said to be \emph{optimal} or \emph{100\% informationally efficient}. In contrast with the probabilistic setting, this framework allows suboptimal IPRs to genuinely lose information or introduce extraneous information.

\paragraph{Information imbalance functional.} To operationalise the ICP, we define the information imbalance induced by an IPR as the function:
\begin{align*}
\Delta(\text{IPR}) & \doteq \underbrace{\big( \mathcal{I}[\pi] + \mathcal{I}[L(\cdot \given y)] \big)}_{\text{input information}}
   - \underbrace{\big( \mathcal{I}[g(\cdot \given y)] + \mathcal{I}[c(y)] \big)}_{\text{output information}}
\end{align*}
that is, using $L(\cdot\given y) = h_y^{w}$ from~\eqref{eq:BayesByCombination},
\begin{multline}
\Delta(\text{IPR})(\theta) =
-\log \pi(\theta)
+ w\big(\ell(\theta,y)-m(y)\big) \\
+\log g(\theta \given y)
+\log c(y),
\end{multline}
where $c(y)$ is constant in $\theta$, so that $\mathcal{I}[c(y)] = -\log c(y)$.
The strength enters only through the evidence term $\mathcal I[L]=w(\ell-m)$,
while the global unit $k=1$ in $\mathcal I$ is untouched.
Unlike the probabilistic case, $\Delta(\text{IPR})$ is not a scalar but a function. Consequently, information loss and information creation are to be understood in terms of the pointwise (partial) order:
\[
\Delta(\text{IPR}) \succeq 0 \;\; \Longleftrightarrow \;\;
\forall \theta \in \Theta,\; \Delta(\text{IPR})(\theta) \ge 0.
\]
An IPR is said to conserve information exactly if it holds that $\Delta(\text{IPR}) = 0$.

\paragraph{Optimal information processing rule.} For each fixed strength $w$,
solving $\Delta(\text{IPR}) = 0$ pointwise yields a unique postdata possibility
function, the strength-$w$ update
\[
\pi(\theta\given y) = \frac{\exp(-w\,\ell(\theta,y))\,\pi(\theta)}
{\sup_{\vartheta\in\Theta}\exp(-w\,\ell(\vartheta,y))\,\pi(\vartheta)},
\]
which is \eqref{eq:maxBayesRule} with $L(\cdot \given y)=h_y^{w}$ and coincides with Bayes' rule
in possibilistic form. Conservation fixes the posterior once the loss scale is set
but does not itself select $w$, which remains the exogenous strength of
Section~\ref{sec:possibility_theory}. Unit strength $w=1$ recovers the
sup-normalised prior-likelihood product that is the focus of this work.

\paragraph{Information creation and information loss.} An IPR does not create extraneous information if and only if: 
\[ \Delta(\text{IPR}) \succeq 0 \quad \Longleftrightarrow \quad g(\theta \given y) \ge \pi
(\theta \given y) \quad \forall \theta \in \Theta. 
\]
Similarly, an IPR does not lose information if and only if: 
\[ \Delta(\text{IPR}) \preceq 0 \quad \Longleftrightarrow \quad g(\theta \given y) \le \pi
(\theta \given y) \quad \forall \theta \in \Theta. 
\]
Thus, Bayes' rule at strength $w$ is the only IPR that belongs simultaneously to both classes. 

These conditions admit a natural interpretation in terms of possibility theory. If the information about an uncertain variable $\boldsymbol{\theta}$ is described by a possibility function $f$ on $\Theta$, then it is also described by any possibility function $g$ on $\Theta$ such that $g(\theta) \ge f(\theta)$ for all $\theta \in \Theta$. In this case, $g$ is said to be \emph{less informative} than $f$, as it rules out fewer values of $\theta$. Accordingly, postdata possibility functions that dominate the Bayesian posterior everywhere do not introduce information beyond what is justified by the input, whereas postdata possibility functions that are everywhere dominated by the Bayesian posterior necessarily introduce additional information. These two classes do not, however, cover all of $\mathcal{F}(\Theta)$: since $\preceq$ is only partial, a generic postdata $g(\cdot\given y)$ is \emph{incomparable} to $\pi(\cdot\given y)$, with $\Delta(\text{IPR})$ changing sign across $\Theta$, creating information at some parameter values while losing it at others. The possibilistic setting therefore exhibits three qualitatively distinct suboptimal regimes, loss without creation ($g\succeq\pi(\cdot\given y)$), creation without loss ($g\preceq\pi(\cdot\given y)$), and mixed incomparability, with Bayes' rule the unique point of exact conservation. Far from a defect, this is the payoff of the pointwise order: it localises \emph{where} a rule adds or discards information, which a scalar summary cannot express.

\paragraph{Summary.} Within this possibilistic framework, at any fixed strength
$w$ (writing $\pi(\cdot\given y)$ for the corresponding strength-$w$ update):
\begin{enumerate*}[label=\roman*)]
  \item information loss and information creation are meaningful and observable phenomena that a rule may exhibit separately or, when its postdata belief is incomparable to $\pi(\cdot\given y)$, together at different parameter values,
  \item IPRs that upper-bound $\pi(\cdot\given y)$ do not create information,
  \item IPRs that lower-bound $\pi(\cdot\given y)$ do not lose information, and
  \item Bayes' rule at strength $w$ is the unique IPR that conserves information exactly.
\end{enumerate*}

This stands in sharp contrast with the probabilistic framework where such distinctions are precluded by the scalar nature of the information measure.

\subsection{Connection to maxitive Donsker--Varadhan duality}
\label{sec:dual-guarantees}

To make the dependence of the imbalance $\Delta(\mathrm{IPR})$ on the postdata possibility function $g(\cdot \given y)$ more explicit, we write $\Delta_g(\theta)$ instead of $\Delta(\mathrm{IPR})(\theta)$. At unit strength, the setting of the concurrent maxitive Donsker--Varadhan duality of \citet{singh2025maxitive}, the imbalance and their consistency-bound integrand are two readings of the same function: writing $Z_{\max} \doteq \sup_{\vartheta\in\Theta} \pi(\vartheta) \exp(-\ell(\vartheta,y))$ for their maxitive marginal likelihood, so that $\log c(y) = m(y) + \log Z_{\max}$, it holds that
\[
\Delta_g(\theta) = \log Z_{\max} - \Big( -\ell(\theta,y) - \log\frac{g(\theta \given y)}{\pi(\theta)} \Big).
\]
Their lower and upper consistency bounds are the infimum and supremum over $\theta$ of the bracketed integrand, so the correspondence holds at the level of solution sets: an IPR loses no
information ($\Delta_g \preceq 0$, i.e.\ $g \preceq \pi(\cdot\given y)$) precisely when its postdata belief maximises their lower bound, and creates no information ($\Delta_g \succeq 0$, i.e.\ $g \succeq \pi(\cdot\given y)$) precisely when it minimises their upper bound. The two-sided slack is their max-relative entropy, $\sup_{\theta} \Delta_g = D_{\max}(g \,\|\,
\pi(\cdot\given y))$ and $-\inf_{\theta} \Delta_g = D_{\max}(\pi(\cdot\given y) \,\|\, g)$, and their identification of the posterior as the unique common optimiser of the two bounds is the statement that Bayes' rule is the unique rule incurring neither loss nor creation. For a general strength, applying their formula to the rescaled loss $w\ell$ traces the one-parameter family of Theorem~\ref{thm:max-gibbs}.

\section{View 2: Axiomatic approach}
\label{sec:5}

In the probabilistic case, considering a function of the loss as a likelihood is a departure from the standard Bayesian paradigm whenever the loss is not based on the true DGM, hence the name \emph{generalised} Bayesian inference. However, possibilistic inference does not require the likelihood to be the true DGM and any \emph{information} about the relationship between an observation and the parameter can be encoded as a possibilistic likelihood whenever the loss is bounded below. For consistency, we follow \citet{bissiri2016general} and consider a loss $\ell(\theta, y)$ but highlight that a possibilistic likelihood $h_y(\theta) \propto \exp(-\ell(\theta, y))$ could be considered instead whenever $\inf_{\theta} \ell(\theta, y) > -\infty$.

In the following set of assumptions, we simply write $\ell(\theta)$ for the loss $\ell(\theta, y)$ at a data point $y$. Similarly, we write $\ell'(\theta)$ for the same loss evaluated at another data point $y'$, i.e., $\ell'(\theta) \doteq \ell(\theta, y')$.

\begin{assumptions*}[Possibilistic coherence axioms]
Let $\psi[\ell,\pi]:\Theta\to[0,1]$ denote the posterior possibility function based on the loss $\ell : \Theta \to \mathbb{R}$ and the prior possibility function $\pi$. Assume:
\begin{enumerate}[label=A2.\arabic*,leftmargin=25pt]
\item\label{A1}
It holds that $\psi\big[\ell', \psi[\ell,\pi]\big] = \psi\big[\ell+\ell', \pi\big]$.
\item\label{A2}
For any $A\subseteq\Theta$ with $\sup_{\theta\in A} \pi(\theta)>0$,
\[
\frac{\psi[\ell,\pi](\theta)}{\sup_{\vartheta \in A}\psi[\ell,\pi](\vartheta)} = \psi[\ell,\pi_A](\theta),\qquad \theta \in A,
\]
where $\pi_A(\theta)=\bm{1}_A(\theta) \pi(\theta)\big/\sup_{\vartheta \in A} \pi(\vartheta)$.
\item\label{A3} If $A\subseteq\Theta$ satisfies $\sup_{\theta\in A}\pi(\theta)>0$, $\ell(\theta) > \ell'(\theta)$ for all $\theta \in A$, and $\ell(\theta) = \ell'(\theta)$ for all
$\theta \notin A$, then
\[
\sup_{\theta\in A}\psi[\ell,\pi](\theta) < \sup_{\theta\in A}\psi[\ell',\pi](\theta).
\]
\item\label{A4}
If $\ell$ is constant, then $\psi[\ell,\pi] = \pi$.
\item\label{A5}
If $\tilde{\ell}(\theta) = \ell(\theta) + c$ for some constant $c\in\mathbb{R}$, then $\psi[\tilde{\ell},\pi] = \psi[\ell,\pi]$.
\end{enumerate}
\end{assumptions*}

The axioms respectively impose sequential coherence, restriction compatibility,
monotonicity in loss, invariance to vacuous losses, and invariance to additive
loss shifts. An axiom-by-axiom discussion is given in Appendix~\ref{app:axiom-audit}. Although these axioms closely mirror the coherence logic of \citet{bissiri2016general}, the point here is that this argument is not special to probability. It is not, however, a mechanical integral-to-supremum substitution. The axioms are the possibilistic counterparts of Bissiri's coherence conditions, not a strengthening of them; the extra work lies in the proof. Because a possibility function is fixed by its pairwise odds only up to the supremum constraint, the probabilistic normalisation cannot simply be reused, and one must instead establish, from the same axioms, that maxitive restriction (\ref{A2}) preserves pairwise odds, that binary possibility functions are recovered from those odds, and that the binary updates cohere on triples (Appendix~\ref{sec:proof:max-gibbs}).

\begin{theorem}
\label{thm:max-gibbs}
Under Assumptions~\ref{A1}--\ref{A5}, and assuming that $\Theta$ contains at least three elements, then there exists $w > 0$ such that for every loss $\ell(\cdot,y)$ satisfying $0<\sup_{\vartheta\in\Theta}\exp(-w\ell(\vartheta,y))\pi(\vartheta)<\infty$, and prior possibility function $\pi$, it holds that
\[
\psi[\ell,\pi](\theta) =
\frac{\exp(-w \ell(\theta, y)) \pi(\theta)}
{\sup_{\vartheta \in \Theta} \exp(-w \ell(\vartheta, y)) \pi(\vartheta)}.
\]
\end{theorem}

The proof is deferred to Appendix~\ref{sec:proof:max-gibbs}. The theorem leaves $w$ undetermined since rescaling $\ell\mapsto a\ell$ reproduces the family with $aw$. This is consistent with its status as the exogenous strength of Section~\ref{sec:possibility_theory}. When the prior itself is constructed from a regulariser, its strength is a second modelling input rather than a quantity determined by the formalism. In that case the relative prior-loss weight controls the posterior mode, while the overall strength controls the contraction of $\alpha$-cuts. Neither is selected by possibilistic evidence alone. Details are given in Appendix~\ref{app:losses-regularisers}.

\paragraph{Convergence of the two views.} The two derivations do not merely share a functional form; they yield the \emph{same} one-parameter family of posteriors. For any loss $\ell(\theta,y)$ and any strength $w\ge0$, View~1 encodes the loss as the evidence possibility $h_y^{w}(\theta)=\exp(-w(\ell(\theta,y)-m(y)))$ and, through information conservation, returns the strength-$w$ update $\pi(\theta\given y)\propto\pi(\theta)\exp(-w\ell(\theta,y))$ of~\eqref{eq:BayesByCombination}, while View~2 derives the identical family from the coherence axioms in Theorem~\ref{thm:max-gibbs}. The agreement is structural rather than incidental: in both views the exponential form traces to the same fact, i.e., independent information accumulates additively on the log scale, which enters once as additivity of the pointwise information score~(\ref{A1.3}) and once as sequential coherence of updating~(\ref{A1}). Both are instances of Cauchy's functional equation. Its monotone solution is the logarithm when it fixes the information score in View~1 and the exponential when it fixes the loss-to-possibility map in View~2. These are mutually inverse,
which is precisely why the two derivations compose into a single rule.

What the two formalisms make explicit is that the strength $w$ is exogenous to \emph{both} derivations, though it enters differently. In View~1, conservation is transparent to scale: it fixes the posterior completely once $w$ is chosen, but $w$ is set upstream, in the encoding of the loss as $h_y^{w}$, and conservation itself never selects it. In View~2, $w$ is the constant left undetermined by the coherence axioms. Neither principle pins it down, consistent with its status as the non-identifiable epistemic-strength parameter of Section~\ref{sec:possibility_theory} and Appendix~\ref{app:losses-regularisers}.

\paragraph{Continuous parameter spaces.} In standard Bayesian inference, extending consistency results from discrete to continuous spaces often requires measure-theoretic tools such as Radon--Nikodym derivatives because individual points in a continuous space typically have probability mass zero. Since possibility functions are defined pointwise as degrees of possibility rather than probability densities, the pairwise consistency arguments used in the proofs extend directly to uncountably infinite $\Theta$ without measure-theoretic ambiguity. In addition, while the derivation relies on ratios $\pi(a)/\pi(b)$, which technically requires $\pi(b) > 0$, the resulting product form $\psi[\ell,\pi](\theta) \propto \exp(-w \ell(\theta)) \pi(\theta)$ remains valid on the entire domain. If $\pi(\theta) = 0$ (prior impossibility), the posterior correctly vanishes.

To see the correspondence with the familiar Bayesian naming convention, consider the log-loss $\ell(\theta,y) = -\log p_{\theta}(y)$ (whenever $p_\theta(y)>0$). Then Theorem~\ref{thm:max-gibbs} yields the tempered Bayes-like update
\begin{align}
    \pi(\theta\given y) = \frac{p_\theta(y)^{w} \pi(\theta)}{\sup_{\vartheta\in\Theta} p_\vartheta(y)^{w} \pi(\vartheta)}.
\end{align}
For $w=1$ this reduces exactly to the sup-normalised prior-likelihood product which is the focus of this work.

\begin{remark}[Connection with Martin's sup-normalised prior-likelihood baseline]
\label{rem:martin-connection}
Theorem~\ref{thm:max-gibbs} identifies a broad class of coherent possibilistic updating rules and shows that, under Assumptions~\ref{A1}--\ref{A5}, the posterior possibility function must take a \emph{sup-normalised product} form. This agrees with the observation of \citet{martin2022valid} (in the context of \emph{validification} via outer consonant approximation) that supremum normalisation is the appropriate way to convert a non-negative prior-likelihood score into a posterior \emph{possibility function}: unlike integral- or Choquet-type normalisations, dividing by the supremum enforces the defining constraint $\sup_{\theta\in\Theta}\pi(\theta\given y)=1$ for each fixed dataset $y$. Theorem~\ref{thm:max-gibbs} therefore extends this observation: it shows that sup-normalisation is not merely a convenient choice tied to a specific construction, but is forced by general coherence requirements on possibilistic updating.
\end{remark}

\section{Discussion}

This paper characterises a canonical Bayes-type update for possibilistic inference. Under both an information-conservation viewpoint and an axiomatic viewpoint, the posterior is the prior-likelihood product with supremum normalisation, and the two derivations yield the same one-parameter family for an arbitrary loss. The result also highlights that the learning rate is exogenous to both derivations: it governs epistemic strength and cannot, in general, be learned from the evidence alone.

\paragraph{Connection to imprecise probability.} Although the posterior is a possibility function, genuine upper probabilities surface in prediction under a probabilistic likelihood $p_\theta$: the posterior predictive
\[
\bar{P}(B \given y) = \sup_{\theta} \pi(\theta\given y) \int \bm{1}_B(y') p_{\theta}(y') \mathrm{d}y'
\]
is an upper probability, whose use in decisions invokes standard imprecise-probability tools such as Choquet integrals or credal sets.

While our focus is conceptual, the update $\pi(\theta\given y)\propto \pi(\theta) L(\theta\given y)$ can be used directly as a postdata \emph{compatibility score} for inference and prediction. Its $\alpha$-cuts $\{\theta:\pi(\theta\given y)\ge \alpha\}$ define nested regions that can be pushed through a forward map to obtain predictive possibility bands, and summarised either by level-set estimators (e.g., MAP/profiles) or by decision rules that are robust on cuts (optimising worst-case loss over a chosen $\alpha$-cut). Computationally, working on the log scale turns the update into the addition of a prior penalty and a data-fit term, and sup-normalisation replaces intractable integrals by maximisation. Finally, the epistemic-strength parameter (e.g. $w$ in a Gibbs form) is naturally selected by \emph{calibration} rather than by marginal evidence, e.g., by enforcing coverage/validity constraints in the spirit of plausibility or inferential-model constructions \citep{martin2026possibilistic}, or via conformal-style calibration of prediction sets derived from $\pi(\cdot\given y)$.

\paragraph{Information deletion.} Another natural direction concerns the inverse problem of information processing. Whereas the present work asks how new information should be incorporated into an existing state of knowledge, one may equally ask how previously acquired information should be removed. In the probabilistic setting, this question has recently attracted considerable attention under the name of Bayesian unlearning. Beyond its practical motivations, recent work on \emph{optimal information deletion} by \cite{montcho2026optimal} has shown that Bayesian unlearning itself admits a principled information-processing interpretation, providing a conceptual counterpart to the role played by Bayes' theorem for information acquisition.

This perspective suggests a natural extension of the present work. Rather than viewing possibilistic unlearning as merely an algorithmic question, it would be interesting to investigate whether it can itself be characterised from first principles. In particular, one may ask whether a canonical possibilistic unlearning rule can be derived from an optimal-information-deletion principle mirroring the information-conservation viewpoint developed in Section~\ref{sec:4}, or equivalently from an axiomatic characterisation paralleling the coherence arguments of Section~\ref{sec:5}. Such a result would provide a possibilistic counterpart to recent developments in probabilistic unlearning, while further testing the scope of the foundational principles developed here.

\paragraph{Future avenues.} This paper does not aim to develop a full inferential or decision-making pipeline, nor to provide empirical comparisons; our contribution is to clarify the informational principles and coherence requirements that single out particular update rules within a broader family of ``Bayes' rules.'' Important next steps include
\begin{enumerate*}[label=\roman*)]
\item systematic calibration strategies for the epistemic-strength parameter (e.g.\ $w$) based on coverage/validity or prequential criteria,
\item the study of predictive set constructions obtained by propagating $\alpha$-cuts through complex models, and
\item characterising how alternative conjunction operators (beyond the product t-norm) induce different update rules and different operational guarantees.
\end{enumerate*}
These directions would connect the present foundations to concrete algorithms and benchmarks.

\begin{acknowledgements} 
This research is supported by the Singapore Ministry of Digital Development and Information under the AI Visiting Professorship Programme (AIVP-2024-004).
BECA acknowledges funding from the ANR grant project BACKUP ANR-23-CE40-0018-01.
\end{acknowledgements}

% References
\bibliography{References}

\newpage

\onecolumn

\title{From Bayes' Rule to Bayes Rules: \\
Optimal Information Processing and Axiomatic Foundations Beyond Probability \\(Supplementary Material)}
\maketitle

\appendix

\section{Losses, regularisers, and posterior concentration}
\label{app:losses-regularisers}

Suppose the prior encodes a regulariser $R:\Theta\to[0,\infty)$ (with $\inf_\theta R=0$) and the evidence a loss, so that, writing
$\tilde\ell(\theta)\doteq\ell(\theta,y)-m(y)$, Theorem~\ref{thm:max-gibbs} gives the sup-normalised posterior $\pi(\theta\given y)\propto\exp(-[\beta R(\theta)+w\tilde\ell(\theta)])$ at strengths $\beta,w\ge0$. Reparametrise by the relative weight $\lambda\doteq\beta/w$ and the overall strength $s\doteq w$, and let $Q_\lambda\doteq\lambda R+\tilde\ell$. Then
\[
  \pi(\theta\given y)
  =\exp\!\big(-s\,[Q_\lambda(\theta)-\textstyle\inf_\vartheta Q_\lambda(\vartheta)]\big)
  =\big[\pi^{(1)}_\lambda(\theta\given y)\big]^{s},
\]
the $s$-th power of the unit-strength posterior determined by $\lambda$ alone. This highlights two distinct degrees of freedom. The \emph{location} of the posterior, in particular its mode $\arg\min_\theta Q_\lambda$, depends only on the ratio $\lambda$, since scaling by $s > 0$ does not move the $\arg\min$. This is the regime of penalised optimisation, where a loss and a regulariser are combined and only their relative weight is meaningful. The \emph{concentration}, however, depends on $s$: for $\alpha\in(0,1)$ the $\alpha$-cut $\{\theta:\pi(\theta\given y)\ge\alpha\} = \{\theta:Q_\lambda(\theta)-\inf_\vartheta Q_\lambda\le -s^{-1}\log\alpha\}$ contracts as $s$ increases. The possibilistic posterior therefore contains information that is not relevant for penalised optimisation: optimisation focuses on one aspect ($\lambda$), possibilistic updating on two ($\lambda$ and $s$). A special case is $\beta=w\neq1$, i.e.\ $\lambda=1$ with $s\neq1$: the penalised objective $R+\tilde{\ell}$ and its minimiser are unaffected, yet the posterior is tempered.

This also answers whether attaching a strength to each factor is the same as choosing a different constant $k$ in the pointwise information score $\mathcal{I}[\cdot]=-k\log(\cdot)$ of View~1. The answer is positive for the inputs: measuring the prior with $k_\pi=\beta$ and the loss-evidence with $k_L=w$ reproduces the factors $\pi^\beta$ and $h_y^{w}$. However, a \emph{common} $k$ appears as a multiplicative factor in every term of $\Delta(\mathrm{IPR})$ and cancels in the conservation identity $\Delta(\mathrm{IPR})=0$, which is why View~1 fixes the update without fixing any scale. \emph{Per-factor} constants differ across that identity and do not cancel. The obtained differences are exactly the parameters $(\lambda,s)$ above.

\section{Example: Squared-loss regression under misspecification}
\label{sec:sq-loss-misspec}

Let $D_n=\{(x_i,y_i)\}_{i=1}^n$ with fixed design $x_i\in\mathbb{R}$ and data generated from a misspecified model
\[
y_i = f_0(x_i)+\varepsilon_i,\qquad f_0(x)=\beta x+\gamma x^2,\qquad \varepsilon_i\stackrel{\text{iid}}{\sim}N(0,\sigma^2),
\]
but we fit the linear working model $m_\theta(x)=\theta x$ using the squared loss
\[
\ell_n(\theta) \doteq \sum_{i=1}^n (y_i-\theta x_i)^2.
\]
Write the ``Gibbs evidence'' $L_w(\theta\given D_n) = \exp(-w \ell_n(\theta))$, where $w>0$ controls the epistemic strength of the data relative to the prior.

\paragraph{Generalised Bayes (integral normalisation).}
Take a Gaussian prior $p_0(\theta)\propto\exp\{-(\theta-\mu_0)^2/(2\tau^2)\}$. The generalised posterior is
\[
p_w(\theta\given D_n) \propto p_0(\theta)\exp\{-w \ell_n(\theta)\}.
\]
Since $\ell_n(\theta)$ is quadratic in $\theta$, $p_w(\theta\given D_n)$ is Gaussian with precision
\[
\Lambda_w = \frac{1}{\tau^2}+2w\sum_{i=1}^n x_i^2,
\qquad
v_w = \Lambda_w^{-1},
\qquad
m_w = v_w\Bigl(\frac{\mu_0}{\tau^2}+2w\sum_{i=1}^n x_i y_i\Bigr).
\]
Under misspecification, $m_w$ still targets the best linear approximation to $f_0$ in the design-weighted $L^2$ sense, while the spread $v_w$ depends on $w$ and does not correspond to a correct noise variance unless $w$ is calibrated.

\paragraph{Possibilistic Bayes (sup normalisation).}
Take the corresponding Gaussian-shaped \emph{prior possibility} $\pi_0(\theta)=\exp\{-(\theta-\mu_0)^2/(2\tau^2)\}$. The possibilistic update is
\[
\pi_w(\theta\given D_n) = \frac{\pi_0(\theta)\exp\{-w \ell_n(\theta)\}}{\sup_{\vartheta}\pi_0(\vartheta)\exp\{-w \ell_n(\vartheta)\}}
= \exp \Bigl(-\bigl[Q_w(\theta)-\inf_{\vartheta}Q_w(\vartheta)\bigr]\Bigr),
\]
where $Q_w(\theta) \doteq (\theta-\mu_0)^2/(2\tau^2) + w \ell_n(\theta)$. Hence $\pi_w(\cdot\given D_n)$ is a log-quadratic possibility function with the same mode $\hat{\theta}_w = \arg\min_{\theta} Q_w(\theta)$ as the generalised Bayes posterior, but it is normalised by a supremum rather than an integral.

\paragraph{Interpretable uncertainty via $\alpha$-cuts.}
For any $\alpha\in(0,1]$, the $\alpha$-cut is the interval
\[
\{\theta:\pi_w(\theta\given D_n)\ge\alpha\}
= \{\theta:Q_w(\theta)\le Q_w(\hat\theta_w)-\log\alpha\}
= \biggl[\hat\theta_w\pm \sqrt{\frac{-\log\alpha}{a_w}}\biggr],
\]
where $a_w = \frac{1}{2}\frac{d^2}{d\theta^2}Q_w(\theta) = \frac{1}{2\tau^2}+w\sum_{i=1}^n x_i^2$ (constant in $\theta$). These level-set regions are likelihood-ratio-like and avoid any appeal to a correct sampling likelihood.

\paragraph{Role of misspecification and calibration of $w$.}
In both approaches $w$ controls concentration: larger $w$ yields narrower $p_w(\theta\given D_n)$ and tighter $\alpha$-cuts. However, in the possibilistic update the normalising constant
\[
c_w(D_n) \doteq \sup_{\theta}\pi_0(\theta)\exp\{-w\,\ell_n(\theta)\}
= \exp\{-\inf_\theta Q_w(\theta)\}
\]
is non-increasing in $w$, so ``evidence maximisation'' pushes $w$ to the smallest admissible value. This highlights that $w$ should be chosen by \emph{calibration} (for example, selecting $w$ so that predictive $\alpha$-cuts achieve a desired coverage level on held-out data or via a conformal-style procedure), rather than by possibilistic evidence alone.

\section{Bayes' Rule as Optimal Information Processing (Probability)}
\label{app:optimal_information_processing}

\citet{zellner1988optimal} formulates statistical inference as an \emph{information processing} problem.
An \emph{information processing rule} (IPR) transforms given \emph{input information} (prior distribution and likelihood) into \emph{output information} (postdata distribution and evidence).
Different IPRs may lose information or introduce extraneous information.

\subsection*{Setup}

Let $\Theta$ be a parameter space and let $\theta\in\Theta$ be an unknown parameter. Let $y$ denote the observed data, with likelihood
\[
L(\theta\given y)=p_\theta(y),
\]
and let $\pi(\theta)$ be the prior (antedata) density. An IPR outputs a postdata density $q(\theta\given y)$ and the evidence; we do not assume a priori that $q(\cdot\given y)$ is given by Bayes' rule. The evidence is
\[
p(y)=\int_\Theta \pi(\theta)p_\theta(y)d\theta.
\]

\subsection*{Information Measures}

Information is quantified via Shannon's cross-entropy, i.e.\ as expectations of minus log-densities with respect to the postdata pdf $q(\theta \given y)$:
\begin{align*}
\mathcal{I}\left[\pi\right] &= \mathbb{E}_{\theta\sim q(\theta\given y)}\left[- \log \pi(\theta)\right] \, , \\
\mathcal{I}\left[L(\cdot \given y)\right] &= \mathbb{E}_{\theta\sim q(\theta\given y)}\left[- \log L(\theta \given y)\right] \, , \\
\mathcal{I}\left[q(\cdot\given y)\right] &= \mathbb{E}_{\theta\sim q(\theta\given y)}\left[- \log q(\theta \given y)\right] \, , \\
\mathcal{I}[p(y)] &= -\log p(y).
\end{align*}
Input information consists of $\mathcal{I}\left[\pi\right]+\mathcal{I}\left[L(\cdot \given y)\right]$, and output information consists of $\mathcal{I}\left[q(\cdot\given y)\right] + \mathcal{I}[p(y)]$.

\subsection*{Information Conservation Principle (ICP)}

\emph{A good IPR should conserve information:}
\[
\text{Input information} = \text{Output information}.
\]
An IPR satisfying this condition is said to be either \emph{optimal} or \emph{100\% informationally efficient}.
Loss of information or creation of extraneous information is undesirable.

\subsection*{Criterion Functional}

To operationalise the ICP, Zellner defines the \emph{information loss} criterion
\begin{align*}
\Delta(\mathrm{IPR})
& = \text{Input information}(\text{IPR}) - \text{Output information}(\text{IPR}) \\
& = \mathrlap{\left(\mathcal{I}\left[\pi\right]+\mathcal{I}\left[L(\cdot\given y)\right]\right)} \hphantom{\text{Input information}(\text{IPR})}
 -\left(\mathcal{I}\left[q(\cdot\given y)\right]+\mathcal{I}\left[p(y)\right]\right)\\
&=\mathbb{E}_{\theta\sim q(\cdot\given y)}\left[-\log\pi(\theta)-\log L(\theta \given y)+\log q(\theta\given y)+\log p(y)\right]\\
&= \mathbb{E}_{\theta\sim q(\cdot\given y)}\left[- \log L(\theta \given y)\right] + \mathrm{KL}\big(q(\cdot\given y)\|\pi\big)+\log p(y) .
\end{align*}
This criterion measures the difference between input and output information.
An optimal IPR is a zero of $\Delta(\cdot)$.

\subsection*{Optimal Information Processing Rule}

The unique optimal IPR is
\[
\pi_{\textrm{add}}(\theta \given y)
= \frac{\pi(\theta)\, L(\theta \given y)}
{\int \pi(\theta)\, L(\theta \given y)\, d\theta} \, ,
\]
which is exactly obtained via \emph{Bayes' rule}.

\vspace{0.2cm}

Bayes' rule is derived as the \emph{optimal information processing rule} for the chosen information measures. The analysis links Bayesian updating to entropy and relative entropy concepts, showing that Bayesian inference neither loses nor creates information relative to the specified inputs.

\subsection*{Sketch of the proof}

The proof is in fact very simple. We just need to notice that
$$
\Delta(\text{IPR}) = \mathrm{KL}\left(q(\cdot \given y) \| \pi_{\textrm{add}}(\cdot \given y) \right) \, ,
$$
which is equal to $0$ only when $q(\cdot \given y)=\pi_{\textrm{add}}(\cdot \given y)$, which is given by Bayes' rule.

\subsection*{Statistical inference as Optimisation}

When adopting an information processing perspective on statistical inference, we remark that the optimal IPR (Bayes' rule) can be defined as \emph{the single information loss minimiser} (since no IPR creates information under Zellner's information measures). This provides the standard variational formulation of (additive) Bayesian inference:
$$
\big\{ \pi_{\textrm{add}}(\cdot \given y) \big\} = \arg\min_{q(\cdot\given y)} \, \bigg\{ \mathbb{E}_{\theta\sim q(\cdot\given y)}\left[- \log L(\theta \given y)\right] + \mathbb{E}_{\theta\sim q(\cdot\given y)}\left[\log \frac{q(\theta\given y)}{\pi(\theta)}\right] \bigg\} \, ,
$$
which directly follows from rewriting the information loss quantity as:
\begin{align*}
\Delta(\text{IPR}) & = \text{Input information}(\text{IPR}) - \text{Output information}(\text{IPR}) \\
& = \mathbb{E}_{\theta\sim q(\cdot\given y)}\left[- \log L(\theta \given y)\right] + \mathrm{KL}\big(q(\cdot\given y)\|\pi\big)+\log p(y) .
\end{align*}

Note that in the possibilistic framework considered in this work, the optimal IPR is no longer defined via a single optimisation problem but as the intersection of two complementary optimisation problems:
$$
\{ \pi(\cdot\given y) \}
= \arg\max_{g(\cdot\given y) \in \mathcal F(\Theta)} \inf_{\theta \in \Theta}
\left\{ -\ell(\theta,y) - \log \frac{g(\theta\given y)}{\pi(\theta)} \right\}
\;\cap\;
\arg\min_{g(\cdot\given y) \in \mathcal F(\Theta)} \sup_{\theta \in \Theta}
\left\{ -\ell(\theta,y) - \log \frac{g(\theta\given y)}{\pi(\theta)} \right\} ,
$$
where each problem is solved on its own by an entire class of postdata possibility functions, those below and those above $\pi(\cdot\given y)$ respectively, and only their intersection is a singleton.

\section{Axiom audit}
\label{app:axiom-audit}

We briefly discuss which coherence requirements are responsible for each structural feature of the update and what may fail if an axiom is relaxed.
\begin{enumerate}[label=A2.\arabic*,leftmargin=25pt,wide]
\item (\emph{Sequential coherence}) This axiom enforces that sequential incorporation of two losses is equivalent to incorporating their sum. It is the key ingredient behind the additive structure in the loss scale, and hence the multiplicative structure after exponentiation. Without (\ref{A1}), the update may become path-dependent (order effects) and the representation $\psi[\ell,\pi]\propto \pi\exp(-w\ell)$ need not hold.
\item (\emph{Locality under restriction}) This locality condition is the main driver of the ``odds-separability'' property used in the characterisation. It rules out updates whose action on $\theta$ depends on losses outside the restricted set (e.g., through global normalisations other than $\sup$ or non-local aggregation). If (\ref{A2}) is dropped, one can construct coherent-but-nonlocal updates where the relative update between two points depends on the loss landscape elsewhere, and the exponential form may fail.
\item (\emph{Monotonicity in loss}) This ensures that larger loss leads to (weakly) smaller posterior possibility, yielding $w\ge 0$ in the representation. Without (\ref{A3}), one could admit updates that perversely increase plausibility under worse loss, and the learning-rate sign (and even existence) is no longer controlled.
\item (\emph{Vacuous loss leaves beliefs unchanged}) This pins down the identity element of updating and prevents spurious changes when the loss carries no discriminatory information. Relaxing (\ref{A4}) allows arbitrary distortions even when $\ell$ is constant in $\theta$, breaking the intended interpretation of $\ell$ as information.
\item (\emph{Shift invariance of loss}) This guarantees that only relative losses matter: adding a constant to $\ell$ has no effect. It is responsible for the appearance of a normalising constant that depends on $y$ only, and for the natural $\sup$-normalisation of the posterior possibility function. Dropping (\ref{A5}) permits updates that depend on arbitrary baselines of $\ell$, undermining comparability across datasets and destroying the canonical normalisation.
\end{enumerate}

\section{Proof of Theorem~\ref{thm:max-gibbs}}
\label{sec:proof:max-gibbs}

To avoid division by zero in odds ratios, we first prove the result for strictly positive priors. The extension to priors with zeros follows by applying the result on the positive support and assigning posterior possibility zero wherever $\pi(\theta)=0$. We start with a simple lemma showing that applying Assumption~\ref{A2} on a pair of points in $\Theta$ implies a form of ratio consistency.

\begin{lemma}
\label{lem:pairwise-restrict}
Under Assumption~\ref{A2}, for any $a,b \in \Theta$ such that $a \neq b$, and for any loss $\ell$, it holds that
\[
\frac{\psi[\ell,\pi](a)}{\psi[\ell,\pi](b)} = \frac{\psi[\ell, \pi_{\{a,b\}}](a)}{\psi[\ell, \pi_{\{a,b\}}](b)}.
\]
\end{lemma}

\begin{proof}
Apply Assumption~\ref{A2} with $A=\{a,b\}$. Dividing the two equalities for $\theta = a$ and $\theta = b$ cancels the common factor
$\sup_{\vartheta \in A}\psi[\ell,\pi](\vartheta)$, from which we can conclude.
\end{proof}

The next lemma shows how to recover a possibility function from odds in the binary case.

\begin{lemma}
\label{lem:binary-odds}
Let $\Theta = \{0,1\}$ and $f$ be a possibility function and define $t=f(0)/f(1)\in(0,\infty)$, then
\[
f(0)=t\land 1,\qquad f(1)=t^{-1}\land 1.
\]
\end{lemma}

\begin{proof}
If $t\ge 1$ then $f(0)\ge f(1)$ so $f(0)=1$ and $f(1)=1/t$, hence $f(0)=t\land 1$ and $f(1)=t^{-1}\land 1$.
If $t\le 1$ then $f(1)=1$ and $f(0)=t$, giving the same formulas.
\end{proof}

The next lemma shows how the binary odds update reduces to a two-argument map.

\begin{lemma}
\label{lem:phi-exists}
Under Assumptions~\ref{A1},\ref{A4}, and \ref{A5}, let $\Theta=\{a,b\}$, let $\Delta = \ell(a)-\ell(b)$ and let $t = \pi(a)/\pi(b)$. Then there exists a function $\Phi : \mathbb{R}\times(0,\infty) \to (0,\infty)$ such that
\[
\frac{\psi[\ell,\pi](a)}{\psi[\ell,\pi](b)} = \Phi(\Delta,t),
\]
and $\Phi$ satisfies $\Phi(0,t) = t$ and, for all $\Delta_1, \Delta_2 \in \mathbb{R}$ and all $t > 0$,
\begin{align*}
\Phi(\Delta_2,\Phi(\Delta_1,t)) = \Phi(\Delta_1+\Delta_2,t).
\end{align*}
Moreover, under Assumption~\ref{A3} the map $\Delta \mapsto \Phi(\Delta,t)$ is strictly decreasing for a given $t$.
\end{lemma}

\begin{proof}
On $\Theta = \{a,b\}$, Assumption~\ref{A5} implies the posterior ratio can only depend on the loss difference $\Delta=\ell(a)-\ell(b)$, not on common shifts. By Lemma~\ref{lem:binary-odds}, the prior is determined (up to normalisation) by $t = \pi(a)/\pi(b)$, so the posterior ratio is a function of $(\Delta,t)$, which we define as $\Phi(\Delta,t)$. If $\Delta = 0$ then $\ell$ is constant, hence by Assumption~\ref{A4}, $\psi[\ell,\pi] = \pi$ and $\Phi(0,t) = \pi(a)/\pi(b) = t$. For coherence, we apply Assumption~\ref{A1} on the 2-point space: updating first by a loss difference $\Delta_1$ and then by $\Delta_2$ yields the same posterior as updating once by $\Delta_1 + \Delta_2$. Indeed, Assumption~\ref{A1} yields
\begin{align*}
\frac{\psi\big[\ell', \psi[\ell,\pi]\big](a)}{\psi\big[\ell', \psi[\ell,\pi]\big](b)} = \frac{\psi\big[\ell+\ell', \pi\big](a)}{\psi\big[\ell+\ell', \pi\big](b)}
\iff \Phi\bigg(\Delta_2, \frac{\psi[\ell,\pi](a)}{\psi[\ell,\pi](b)}\bigg) = \Phi(\Delta_1+\Delta_2,t),
\end{align*}
which implies the stated associativity equation for $\Phi$. Finally, Assumption~\ref{A3} implies that increasing $\ell(a)$ relative to $\ell(b)$ lowers the posterior of $a$ relative to $b$, hence $\Delta \mapsto \Phi(\Delta,t)$ is strictly decreasing.
\end{proof}

% Triple compatibility forces multiplicativity

For regions where the prior assigns total impossibility ($\pi(\theta) = 0$), the posterior is defined as $0$ by taking the limit $t \to 0$ in the odds ratio formulation (Lemma~\ref{lem:phi-exists}), preserving the validity of the product form on the entire domain without violating the coherence axioms.

\begin{lemma}
\label{lem:triple-multiplicative}
Under Assumption~\ref{A2} and considering the setup of Lemma~\ref{lem:phi-exists}, for $\Theta = \{0,1,2\}$, for all $\Delta_1, \Delta_2 \in \mathbb{R}$ and $t_1, t_2 > 0$,
\[
\Phi(\Delta_1 + \Delta_2, t_1 t_2) = \Phi(\Delta_1, t_1) \Phi(\Delta_2, t_2).
\]
Therefore, it holds that $\Phi(\Delta,t) = g(\Delta)t$ with $g(\Delta) \doteq \Phi(\Delta,1)$.
\end{lemma}

\begin{proof}
Let $\Delta_{ab} \doteq \ell(a) - \ell(b)$ and $t_{ab} \doteq \pi(a)/\pi(b)$. Note that $\Delta_{02} = \Delta_{01} + \Delta_{12}$ and $t_{02}=t_{01}t_{12}$. By Lemma~\ref{lem:pairwise-restrict}, for each pair $(a,b)$ the posterior ratio on $\{0,1,2\}$ equals the posterior ratio in the restricted 2-point problem $\{a,b\}$, hence
\[
\frac{\psi[\ell,\pi](a)}{\psi[\ell,\pi](b)} = \Phi(\Delta_{ab},t_{ab}).
\]
Since it holds that
\[
\frac{\psi[\ell,\pi](0)}{\psi[\ell,\pi](2)}
=
\frac{\psi[\ell,\pi](0)}{\psi[\ell,\pi](1)} \cdot \frac{\psi[\ell,\pi](1)}{\psi[\ell,\pi](2)},
\]
substituting the $\Phi$-representations and using the identities for $\Delta_{ab}$ and $t_{ab}$ yields
\[
\Phi(\Delta_{01}+\Delta_{12}, t_{01}t_{12}) = \Phi(\Delta_{01},t_{01}) \Phi(\Delta_{12},t_{12}).
\]
Since $(\Delta_{01},\Delta_{12},t_{01},t_{12})$ can be chosen arbitrarily by varying $\ell$ and $\pi$ on the three points, this proves the first result of the lemma. Setting $t_2 = 1$ and using $\Phi(0,t) = t$ from Lemma~\ref{lem:phi-exists}, we obtain that
\[
\Phi(\Delta,t) = \Phi(\Delta,1)\,\Phi(0,t) = g(\Delta)t,
\]
as required.
\end{proof}

% Exponential form

\begin{lemma}
\label{lem:exponential}
Under Assumptions~\ref{A1}--\ref{A5}, there exists $w > 0$ such that
$g(\Delta)=\exp(-w\Delta)$, so that $\Phi(\Delta,t) = \exp(-w\Delta) t$.
\end{lemma}

\begin{proof}
From Lemma~\ref{lem:triple-multiplicative} and Lemma~\ref{lem:phi-exists},
\[
g(\Delta_1 + \Delta_2) = g(\Delta_1) g(\Delta_2),
\]
with $g(0)=1$. Therefore, $h(\Delta) \doteq \log g(\Delta)$ satisfies Cauchy's equation $h(\Delta_1+\Delta_2) = h(\Delta_1)+h(\Delta_2)$.
The monotonicity of $g$ implies that $h(\Delta) = -w\Delta$ for some $w \in \mathbb{R}$. Since $g$ is decreasing, $w > 0$. Therefore $g(\Delta) = \exp(-w\Delta)$ and $\Phi(\Delta,t) = g(\Delta)t = \exp(-w\Delta)t$ as claimed.
\end{proof}

We are now in a position to prove the main result, extending the previous results to arbitrary sets $\Theta$ via pairwise odds, giving rise to a form of maxitive Gibbs posterior.

\begin{proof}[Proof of Theorem~\ref{thm:max-gibbs}]
Let $a,b\in\Theta$ such that $a \neq b$. By Lemma~\ref{lem:pairwise-restrict}, the posterior ratio between $a$ and $b$ equals the posterior ratio in the restricted 2-point problem $\{a,b\}$. Hence, by Lemmas~\ref{lem:phi-exists},
\ref{lem:triple-multiplicative} and \ref{lem:exponential},
\[
\frac{\psi[\ell,\pi](a)}{\psi[\ell,\pi](b)}
= \exp\big(-w(\ell(a)-\ell(b))\big) \frac{\pi(a)}{\pi(b)}.
\]
Pick an anchor $\theta_0\in\Theta$. Rearranging gives, for all $\theta\in\Theta$,
\[
\psi[\ell,\pi](\theta) = \psi[\ell,\pi](\theta_0) \exp\big(-w(\ell(\theta)-\ell(\theta_0))\big)
\frac{\pi(\theta)}{\pi(\theta_0)}
= C \exp(-w \ell(\theta)) \pi(\theta),
\]
where $C \doteq \psi[\ell,\pi](\theta_0)\exp(w\ell(\theta_0)) / \pi(\theta_0)$ does not depend on $\theta$. Finally, $\psi[\ell,\pi]$ is a possibility function, so $\sup_{\theta} \psi[\ell,\pi](\theta) = 1$. Therefore
\[
1 = \sup_{\theta\in\Theta} \psi[\ell,\pi](\theta)
= C \sup_{\theta\in\Theta} \exp(-w \ell(\theta)) \pi(\theta),
\]
which yields
\[
C^{-1} = \sup_{\theta\in\Theta}\exp(-w \ell(\theta)) \pi(\theta)
\]
and the claimed sup-normalised Gibbs form follows.
\end{proof}

\end{document}